\documentclass[11pt]{amsart}
\usepackage{amsmath,amsthm,amscd,amsfonts, amssymb}
\usepackage{graphicx}
\usepackage{tikz}
\usepackage{color}
\include{diagxy}
\usepgflibrary{shapes}

\font\mbn=msbm10 scaled \magstep1
\font\mbs=msbm7 scaled \magstep1
\font\mbss=msbm5 scaled \magstep1
\newfam\mbff
\textfont\mbff=\mbn
\scriptfont\mbff=\mbs
\scriptscriptfont\mbff=\mbss

\newcommand{\Di}      {\mathbb{D}}

\newcommand{\Z}        {\mathbb{Z}  } 
\newcommand\Co           {{\mathbb C}}

\newtheorem{Th}{Theorem}[section]

\newtheorem{D}[Th]{Definition}
\newtheorem{Proposition}[Th]{Proposition}
\newtheorem{R}[Th]{Remark}

\begin{document}
\title[A Rolle type theorem for cyclicity of zeros]{A ROLLE TYPE THEOREM for CYCLICITY OF ZEROS OF FAMILIES OF ANALYTIC FUNCTIONS
}

\author{Alexander Brudnyi} 
\address{Department of Mathematics and Statistics\newline
\hspace*{1em} University of Calgary\newline
\hspace*{1em} Calgary, Alberta\newline
\hspace*{1em} T2N 1N4}
\email{albru@math.ucalgary.ca}

\thanks{Research supported in part by NSERC}

\begin{abstract}
Let $\{f_{\lambda; j}\}_{\lambda\in V; 1\le j\le k}$ be families of holomorphic functions in the open unit disk $\Di\subset\Co$
depending holomorphically on a parameter $\lambda\in V\subset \Co^n$. We establish a Rolle type theorem
for the generalized multiplicity (called {\em cyclicity}) of zero of the family of univariate holomorphic functions $\left\{\sum_{j=1}^k f_{\lambda;j}\right\}_{\lambda\in V}$ at $0\in\Di$.
As a corollary, we estimate the cyclicity of the family of generalized exponential polynomials, that is, the family of entire functions of the form $\sum_{k=1}^m P_k(z)e^{Q_k(z)}$, $z\in\Co$,
where $P_k$ and $Q_k$ are holomorphic polynomials of degrees $p$ and $q$, respectively, parameterized by vectors of coefficients of $P_k$ and $Q_k$.
\end{abstract}

\date{} 

\maketitle

\section{Introduction}
\subsection{Basic Definitions and Results}
Let $B_n(\lambda, r)\subset\Co^n$ be the open Euclidean ball centered at $\lambda$ of radius $r$. We set $\Di_r := B_1(0, r)\subset\Co$ and $\Di :=\Di_1$. By $K_\varepsilon:=\cup_{\lambda\in K} B_n(\lambda,\varepsilon)$ we denote the $\varepsilon$-neighbourhood of $K\subset \Co^n$. 
Let $V\subset\Co^n$ be a domain and $\mathcal O(V )$ be the
ring of holomorphic functions on $V$. Consider the Maclaurin series expansion of the family $\mathcal F=\{f_\lambda\}_{\lambda\in V}$ of holomorphic functions in the disk $\Di_\rho$ depending holomorphically
on $\lambda\in V$,
\begin{equation}\label{e1.1}
f_\lambda(z)=\sum_{k=0}^\infty a_k(\lambda)z^k,\qquad \lambda\in V,\quad z\in\Di_\rho,\quad a_k\in\mathcal O(V).
\end{equation}
Let $U\subset V$ be a subdomain.
By $\mathcal I(\mathcal F\,;U)\subset \mathcal O(U)$ we denote the ideal generated by all the restrictions $a_k|_{U}$. The central set of $\mathcal F$ in $U$ is defined by the formula
\[
\mathcal C(\mathcal F\,;U):=\{\lambda\in U\, :\, f_\lambda\equiv 0\}.
\]
\begin{D}\label{d1}
The family $\mathcal F=\{f_\lambda\}_{\lambda\in V}$ has cyclicity $k\ge 0$ in a compact set $K\subset V$ if the following holds:
\begin{itemize}
\item[(a)]
There exist $\varepsilon_0>0$ and $\delta_0>0$ such that for every $\lambda$ in $K_{\varepsilon_0}\cap\bigl(V\setminus \mathcal C(\mathcal F;V)\bigr)$
the function $f_\lambda$ has at most $k$ zeros in $\Di_{\delta_0}$ (counting multiplicities).
\item[(b)]
For arbitrary positive numbers $\varepsilon<\varepsilon_0$ and $\delta<\delta_0$ there exists a parameter $\lambda\in K_{\varepsilon}\cap\bigl(V\setminus \mathcal C(\mathcal F;V)\bigr)$ such that $f_\lambda$ has exactly $k$ zeros in $\Di_\delta$.
\end{itemize}
\end{D}
The notion goes back to the famous paper \cite{B1} (see also \cite{B2}) of Bautin who found a connection between cyclicity
and algebraic properties of $\mathcal I(\mathcal F;V)$ and in this way proved that a perturbation
of a center in a family of quadratic vector fields can generate no
more than 3 small amplitude limit cycles.

The cyclicity of the family $\mathcal F$ in $K$ is denoted by $c(\mathcal F; K)$. For a one point
compact set $K :=\{\mu\}$ we use the notation $c(\mathcal F; \mu)$. For $\mu\not\in\mathcal C(\mathcal F; V)$ this number equals the multiplicity
of the zero of $f_\mu$ at $0\in\Di$. Therefore the term ``generalized
multiplicity'' seems to be quite natural for $c(\mathcal F;K)$.

The problem of bounding cyclicity for specific families $\mathcal F$
arises in several
diverse areas of analysis, including dynamical systems (related to the, so-called, second
part of Hilbert's 16 problem), transcendental number theory (related to Hilbert's 7
problem) and approximation theory (Bernstein-Markov-Remez' type inequalities);
for recent developments and corresponding references, see, e.g., \cite{Y} and
the Featured Review [MR 98h:34009]. \medskip

The basic properties of cyclicity are obtained in \cite{Br1}. We formulate some of them.

Let $\mathcal F=\{f_\lambda\}$ and $\mathcal G=\{g_\lambda\}$ be families of holomorphic functions in $\Di_\rho$ 
depending holomorphically on $\lambda\in V$. We set 
\begin{equation}\label{notation}
\mathcal F':=\left\{\frac{df_\lambda}{dz}\right\},\quad \mathcal F+\mathcal G:=\{f_\lambda+g_\lambda\},\quad \mathcal F\cdot\mathcal G:=\{f_\lambda\cdot g_\lambda\}\quad \text{and}\quad e^{\mathcal F}:=\{e^{f_\lambda}\}.
\end{equation}

The following properties are presented in Propositions 2.1, 2.2, 2.5, Theorems 1.5, 1.7 and Remark 1.6 of \cite{Br1} (see also this paper for additional references). 
\begin{itemize}
\item[(a)] 
For a compact set $K\subset V$,
\[
c(\mathcal F;K)=\max_{\mu\in K}c(\mathcal F;\mu).
\]
\item[(b)]
There exist a constant $A$ and an open neighbourhood $U$ of $K$ such that (see \eqref{e1.1}) 
\[
|a_k(\lambda)|\le A\cdot\sup_{U}|a_k|\cdot\max_{0\le i\le c(\mathcal F;K)}|a_i(\lambda)|\quad\text{for all}\quad \lambda\in U;
\]
here $c(\mathcal F;K)$ cannot be replaced by a smaller index.
\end{itemize}
If $K=\{\mu\}$, the latter property is equivalent to the following one (see \cite{HRT}).

Let $\mathcal I(\mathcal F;\mu)$ be the ideal in the local ring $\mathcal O_\mu$ of germs of holomorphic functions at $\mu\in V$ generated by germs at $\mu$ of the Taylor coefficients $a_j$ of $\mathcal F$, see \eqref{e1.1}, and  $\mathcal I_d(\mathcal F;\mu)\subset \mathcal I(\mathcal F;\mu)$ be the ideal generated by germs at $\mu$ of $a_0,\dots, a_d$. We define $d(\mathcal F;\mu)$ to be the minimal integer $d$ such that the integral closures of $\mathcal I_d(\mathcal F;\mu)$ and $\mathcal I(\mathcal F;\mu)$ coincide.
\begin{itemize}
\item[(b$'$)] 
\[
c(\mathcal F;\mu)=d(\mathcal F;\mu).
\]
\item[(c)] 
For a compact set $K\subset V$ there exists an open neighbourhood $O_K$ of $K$ such that the central set $\mathcal C(\mathcal F;O_K)$ coincides with the set of common zeros of $a_0,\dots, a_{c(\mathcal F;K)}$.
\item[(d)] 
\[
c(\mathcal F;K)\le c(\mathcal F';K)+1.
\]
\item[(e)] 
\[
c(e^{\mathcal F};K)=0.
\]
\item[(f)] 
\[
c(\mathcal F\cdot\mathcal G;K)\le c(\mathcal F;K)+c(\mathcal G;K).
\]
\item[(g)] 
\[
c(\mathcal F+\mathcal G;K)\le\max\{c(\mathcal F;K),c(\mathcal G;K), c(W(\mathcal F,\mathcal G);K)+1\};
\]
here $W(\mathcal F,\mathcal G):=\mathcal F'\cdot\mathcal G-\mathcal G'\cdot\mathcal F$ is the family of Wronskian determinants of
$\{\mathcal F,\mathcal G\}$.
\item[(h)]
For every $\varepsilon>0$ there exist an open neighbourhood $O_\varepsilon$ of $K$ and a positive number $R_\varepsilon$ such that for all $\lambda\in O_\varepsilon\setminus\mathcal C(\mathcal F;V)$
\[
\sup_{\Di_{R_\varepsilon}}\ln |f_\lambda|-\sup_{\Di_{R_\varepsilon/e}}\ln |f_\lambda|\le c(\mathcal F;K)+\varepsilon
\]
and each $f_\lambda$ has at most $c(\mathcal F;K)$ zeros in $\Di_{R_\varepsilon}$ (counting multiplicities);

\noindent here $c(\mathcal F;K)$ cannot be replaced by a smaller number.
\end{itemize}
Property (h) generates the following Cartan-type estimates and Remez-type inequalities for the family $\mathcal F$, see, e.g., \cite[Th.~3.1 and Sect.~2]{Br2} and references therein.
\begin{itemize}
\item[(h$'$)]
Fix $H\in (0,1]$, $d > 0$ and set $A:=e^{\left(\frac{e+1}{e-1}\right)^2}$. For each $\lambda\in O_\varepsilon\setminus\mathcal C(\mathcal F;V)$ and $R\in (0,R_\varepsilon)$
there exists a family of open disks $\{D_{j;\lambda}\}_{1\le j\le k}$, $k\le c(\mathcal F;K)$, with
$\sum r_{j;\lambda}^d\le \frac{(2H R)^d}{d}$, where $r_{j;\lambda}$ is the radius of $D_{j;\lambda}$, such that
\begin{equation}\label{cartan}
|f_\lambda(z)|\ge \sup_{\Di_{R}}|f_\lambda|\cdot A^{-\varepsilon}\cdot\left(\frac{H}{A}\right)^{c(\mathcal F;K)}\quad\text{for all}\quad z\in \Di_{R/e}\setminus \cup_j D_{j;\lambda}.
\end{equation}
\item[(h$''$)] Consider the function $\Phi(t):=t+\sqrt{t^2-1}$, $t\ge 1$. There exists an absolute constant $\hat{c}\ge 1$ such that for each $\lambda\in O_1$, an interval $I\subset\Di_{R_1/e}$ and a measurable subset $\omega\subset I$ the following inequality holds:
\begin{equation}\label{remez}
\sup_I |f_\lambda|\le \Phi\left(\frac{2m_1(I)}{m_1(\omega)}-1\right)^{\hat{c}\cdot\max\{c(\mathcal F;K),1\}}\cdot\sup_\omega |f_\lambda|;
\end{equation}
here $m_1$ is the linear measure on $\Co$.
\end{itemize}

\subsection{Formulation of the Main Result}

In this note we generalize property (g) and establish a Role type theorem for the cyclicity of the family 
$\mathcal F:=\sum_{j=1}^k \mathcal F_j$, where $\mathcal F_j:=\{f_{\lambda;j}\}$, $1\le j\le k$, are families of holomorphic functions in $\Di_\rho\subset\Co$
holomorphic in $\lambda\in V\subset \Co^n$.  As a corollary, in the next section we estimate the cyclicity of the
family of generalized exponential polynomials.

In our main result $W(\mathcal F_{i_1},\dots, \mathcal F_{i_l}):=\{W(f_{\lambda;i_1},\dots, f_{\lambda;i_l})\}_{\lambda\in V}$, $1\le i_1<\cdots< i_l\le k$, stands for the family of Wronskian determinants
\[
W(f_{\lambda;i_1},\dots, f_{\lambda;i_l}):=\left|
\begin{array}{cccc}
f_{\lambda;i_1}&f_{\lambda;i_2}&\dots&f_{\lambda;i_l}\\
f_{\lambda; i_1}'&f_{\lambda; i_2}'&\dots&f_{\lambda; i_l}'\\
\vdots&\vdots&\vdots&\vdots\\
f_{\lambda;i_1}^{(l-1)}&f_{\lambda; i_2}^{(l-1)}&\dots&f_{\lambda;i_l}^{(l-1)}
\end{array}
\right|
\]
Let $\mathcal K$ be the family of all distinct subsets of $\{1,\dots, k\}$. By $|I|$ we denote the cardinality of $I\in\mathcal K$. For $I=\{i_1,\dots, i_l\}\in\mathcal K$  and a compact subset $K\subset V$ we set
\[
c(\mathcal F_I;K):=c\bigl(W(\mathcal F_{i_1},\dots, \mathcal F_{i_l});K\bigr).
\]
\begin{Th}\label{main}
\[
c(\mathcal F;K)\le \max_{I\in\mathcal K}\{c(\mathcal F_I;K)+|I|-1\}.
\]
\end{Th}
For $K=\{\mu\}\subset V$ with $\mu\not\in C(\mathcal F;V)$ (i.e., in the case of the multiplicity of zero at $0\in\Di$ of $f_\mu\not\equiv 0$) this result was proved in \cite[Th. 1]{VP}.\medskip

The proof of Theorem \ref{main} is presented in Section 3.
\section{Cyclicity of the Family of Generalized Exponential Polynomials}
Consider the family $\mathcal F=\{f_\lambda\}$ of entire functions of the form 
\[
\sum_{k=1}^m P_k(z)e^{Q_k(z)},\quad z\in\Co,
\]
where $P_k(z)=\sum_{i=0}^p c_{ki}z^i$ and $Q_k(z)=\sum_{j=1}^q d_{kj}z^j$ are holomorphic polynomials of degrees $p$ and $q$, respectively. We parameterize elements of this family by strings $\lambda=\bigl((c_{ki})_{0\le i\le p}, (d_{kj})_{1\le j\le q}\bigr)_{1\le k\le m}$ of the coefficients of polynomials $P_k$ and $Q_k$ involved in their definition, i.e., by points of $\Co^{N_{p,q,m}}$, where $N_{p,q,m}:=m (p+q+1)$.
Then
\begin{equation}\label{series}
f_\lambda(z)=\sum_{n=0}^\infty a_n(\lambda)z^n,\quad z\in\Co,
\end{equation}
where $a_n$ is a holomorphic polynomial in $\lambda$ of degree $n+1$ with nonnegative rational coefficients. Using Taylor series expansion for the exponent together with the multinomial theorem we obtain
\begin{equation}\label{series1}
a_n(\lambda)=\sum_{k=1}^m \left(\sum_{j=0}^{p} c_{kj}\left(\sum_{k_1+2k_2+\cdots +qk_q=n-j}\frac{d_{k k_1}^{k_1} d_{k k_2}^{k_2}\cdots d_{k k_q}^{k_q}}{k_1!\, k_2!\cdots k_q!}  \right)\right).
\end{equation}
(Here the inner summation is taken over all sequences $k_1,\dots, k_q\in\Z_+$ satisfying the required condition.)

\noindent For instance, 
\[
a_0(\lambda)=\sum_{k=1}^m c_{k0},\quad a_1(\lambda)=\sum_{k=1}^m c_{k1}+c_{k0}d_{k1},\quad a_2(\lambda)=\sum_{k=1}^m c_{k2}+c_{k1}d_{k1}+c_{k0}\left(d_{k2}+\frac{d_{k1}^2}{2}\right),
\]
etc.

The center set $\mathcal C(\mathcal F;\Co^{N_{p,q,m}})$ consists of points $\lambda\in \Co^{N_{p,q,m}}$ for which the corresponding generalized exponential polynomials $f_\lambda$ are identically zero. Using a simple induction argument on $m$, based on the comparison of exponential and polynomial growth of entire functions, one obtains that $\sum_{k=1}^m P_ke^{Q_k}\equiv 0$ if and only if there exists a subset $I\subset\{1,\dots, m\}$ such that 
\begin{itemize}
\item[(a)]
$Q_{i}=Q_j$ for all $i,j\in I$, and $Q_i\ne Q_j$ for all distinct $i,j\not\in I$ or if $i\in I$ and $j\not\in I$;
\item[(b)]
$\sum_{k\in I} P_{k}=0$ and $P_k=0$ for all $k\not\in I$.
\end{itemize}

Thus,  $\mathcal C(\mathcal F;\Co^{N_{p,q,m}})=\cup_{I}V_I$, where $I$ runs over all distinct subsets of $\{1,\dots, m\}$ and $V_I\subset\Co^{N_{p,q,m}}$ is the subspace of complex dimension $(|I|-1)(p+1)+(m-|I|+1)q$ defined by conditions (a) and (b); here $|I|$ is the cardinality of $I$.
\begin{R}\label{rem3.1}
{\rm All subspaces $V_I$ with $|I|=1$ coincide, while all other $V_I$ are pairwise distinct; hence, $\mathcal C(\mathcal F;\Co^{N_{p,q,m}})$ is the union of $2^m-m$ pairwise distinct subspaces of $\Co^{N_{p,q,m}}$.}
\end{R}

Let us present other methods related to the subject of the paper describing the center set $\mathcal C(\mathcal F;\Co^{N_{p,q,m}})$.

If $\mathcal F_k=\{f_{\lambda;k}\}$, where $f_{\lambda;k}:=P_ke^{Q_k}$, $1\le k\le m$, so that $\sum_{k=1}^m f_{\lambda;k}=f_\lambda$, then $\lambda\in \mathcal C(\mathcal F;\Co^{N_{p,q,m}})$ if and only if the functions $f_{\lambda;1},\dots ,f_{\lambda;k}$ are linearly dependent over $\Co$, that is, iff the Wronskian determinant $W(f_{\lambda;1},\dots ,f_{\lambda;k})$ is identically zero. According to \cite[Lm.~4, Ex.~1]{VP}, the family $W(\mathcal F_1,\dots,\mathcal F_m)$ has a form $\{S_\lambda e^{T_\lambda}\}_{\lambda\in\Co^{N_{p,q,m}}}$, where $S_\lambda$ and $T_\lambda$ are polynomials of degrees at most $mp+\frac 12 m(m-1)(q-1)$ and $q$, respectively. Thus, due to properties (d), (e) and (f) of cyclicity,
\begin{equation}\label{eq3.8}
c(W(\mathcal F_1,\dots,\mathcal F_m);\mu)\le mp+\frac 12 m(m-1)(q-1)\quad\text{for all}\quad \mu\in \Co^{N_{p,q,m}}.
\end{equation}
Moreover,
\[
\mathcal C(\mathcal F;\Co^{N_{p,q,m}})=\{\lambda\in \Co^{N_{p,q,m}}\, :\, R_\lambda\equiv 0\},
\]
that is, the center set is defined as the zero set of at most $mp+\frac 12 m(m-1)(q-1)$ polynomials in $\lambda$.
 
Next, \eqref{eq3.8} and Theorem \ref{main} imply that {\em the cyclicity of the family $\mathcal F$ of generalized exponential polynomials satisfies the inequality}
\begin{equation}\label{eq3.9}
c(\mathcal F;\mu)\le  m-1+mp+\frac 12 m(m-1)(q-1)\, (=:c_{p,q,m})
\end{equation}
For $\mu\not\in C(\mathcal F;\Co^{N_{p,q,m}})$ (i.e., in the case of the multiplicity of zero at $0\in\Di$ of $f_\mu\not\equiv 0$) this result was proved in \cite{VP}.

Inequality \eqref{eq3.9} and property (c) of cyclicity give an alternative description of the center set of $\mathcal F$ (cf. \eqref{series}, \eqref{series1}):
\[
\mathcal C(\mathcal F;\Co^{N_{p,q,m}})=\{\lambda\in \Co^{N_{p,q,m}}\, :\, a_0(\lambda)=\cdots= a_{c_{p,q,m}}(\lambda)=0\}.
\]

Also, for any relatively compact subset $K\Subset \Co^{N_{p,q,m}}$, there exists a positive number $r_K$ such that each function $f_\lambda$, $\lambda\in K$, has at most $c_{p,q,m}$ zeros in $\Di_{r_K}$ (counting multiplicities).

Let us show that property (b$'$) of cyclicity can be strengthen in our case.
\begin{Proposition}\label{pr3.1}
Integral closures in the ring of holomorphic polynomials on $\Co^{N_{p,q,m}}$
of polynomial ideals $\mathcal I_{c_{p,q,m}}(\mathcal F)$ generated by $a_0,\dots ,a_{c_{p,q,m}}$ and $\mathcal I(\mathcal F)$ generated by all the $a_k$ coincide.
\end{Proposition}
\begin{proof}
Let $\Di^{N_{p,q,m}}$ be the open unit polydisk in $\Co^{N_{p,q,m}}$. According to property (b) of cyclicity, there exists a constant $A>0$ such that for all $k> c_{p,q,m}$ and $\lambda\in \Di^{N_{p,q,m}}$,
\begin{equation}\label{eq3.12}
|a_k(\lambda)|\le A\cdot\sup_{\Di^{N_{p,q,m}}}|a_k|\cdot\max_{0\le i\le c_{p,q,m}}|a_i(\lambda)|.
\end{equation}
Consider a polynomial map $\psi: \Co^{N_{p,q,m}}\rightarrow \Co^{N_{p,q,m}}$ which sends coordinates $c_{ki}$ to $c_{ki}^{i+1}$, $0\le i\le p$,  and coordinates $d_{kj}$ to $d_{kj}^j$, $1\le j\le q$, $1\le k\le m$.  Note that $\psi$ maps $\Di^{N_{p,q,m}}$ onto
$\Di^{N_{p,q,m}}$ and, according to \eqref{series1},
\[
a_k\bigl(\psi(z\cdot\lambda)\bigr)=z^{k+1}\cdot a_k(\psi(\lambda))\quad\text{for all}\quad k\in\Z_+,\ z\in\Co.
\]
These and \eqref{eq3.12} imply, for all $\lambda\in \Co^{N_{p,q,m}}$,
\[
\begin{array}{l}
\displaystyle
|a_k\bigl(\psi(\lambda)\bigr)|\le A\cdot \|\lambda\|_\infty^{k+1}\cdot \sup_{\Di^{N_{p,q,m}}}|a_k|\cdot \max_{0\le i\le c_{p,q,m}}\left\{\frac{|a_i\bigl(\psi(\lambda)\bigr)|}{ \|\lambda\|_\infty^{i+1}}\right\}\medskip \\
\displaystyle
\le A \cdot \bigl(\max\{1,\|\lambda\|_\infty\}\bigr)^k \cdot\sup_{\Di^{N_{p,q,m}}}|a_k|\cdot \max_{0\le i\le c_{p,q,m}}|a_i\bigl(\psi(\lambda)\bigr)|;
\end{array}
\]
here $\|\cdot\|_\infty$ is $\ell_\infty$ norm on $\Co^{N_{p,q,m}}$.

Finally, by the definition of the map $\psi$ we have for all $\lambda\in  \Co^{N_{p,q,m}}$ with $\|\lambda\|_\infty\ge 1$,
\[
\|\lambda\|_\infty=\|\psi^{-1}\bigl(\psi(\lambda)\bigr)\|_\infty\le \|\psi(\lambda)\|_\infty .
\]
Since $\psi$ is surjective, from the previous two formulas we get, for all $\lambda\in  \Co^{N_{p,q,m}}$,
\begin{equation}
|a_k(\lambda)|\le A\cdot\bigl(\max\{1,\|\lambda\|_\infty\}\bigr)^k \cdot\sup_{\Di^{N_{p,q,m}}}|a_k|\cdot\max_{0\le i\le c_{p,q,m}}|a_i(\lambda)|.
\end{equation}\label{eq3.13}
These inequalities together with \cite[Prop.~1.1 (c)]{HRT} imply that each polynomial $a_k$ with $k> c_{p,q,m}$ is integral over the ideal $\mathcal I_{c_{p,q,m}}(\mathcal F)$. This gives the required statement.
\end{proof}
\begin{R}\label{rem3.2}
{\rm 
Proposition \ref{pr3.1} and the  Brian\c{c}on--Skoda-type theorem, see, e.g., \cite[Cor.~13.3.4]{HS}, imply that all coefficients of the Maclaurin series expansion of the family
$\mathcal F^{N_{p,q,m}}=\{f_\lambda^{N_{p,q,m}}\}$ belong to $\mathcal I_{c_{p,q,m}}(\mathcal F)$.
An interesting question is {\em whether the ideal $\mathcal I_{c_{p,q,m}}(\mathcal F)$ is integrally closed?}
}
\end{R}

Next, we present the Cartan-type estimates for the family $\mathcal F$.

For a nonpolynomial $f_\lambda=\sum_{k=1}^m P_k e^{Q_k}$ and $w\in\Co$ by $R_{\lambda;w}$ we denote a (unique) positive number such that
\[
\max_{1\le k\le m}\sup_{\Di_{R_\lambda;w}(w)}|Q_k|=1;
\]
here $\Di_{R}(w):=\{z\in\Co\, :\, |z-w|<R\}$.
\begin{Proposition}\label{prop3.3}
There is a number $R_{\mathcal F}>0$ such that for each nonpolynomial $f_\lambda\in\mathcal F$ and $R\in (0,R_{\lambda;w} R_{\mathcal F})$, and fixed $H\in (0,1]$ and $d > 0$ there exists a family of open disks $\{D_{j;\lambda}\}_{1\le j\le k}$, $k\le c_{p,q,m}$, with
$\sum r_{j;\lambda}^d\le \frac{(2H R)^d}{d}$, where $r_{j;\lambda}$ is the radius of $D_{j;\lambda}$, such that
\begin{equation}\label{cartan1}
|f_\lambda(z)|\ge \sup_{\Di_R(w)}|f_\lambda|\cdot \left(\frac{H}{A}\right)^{c_{p,q,m}+1}\quad\text{for all}\quad z\in \Di_{R/e}(w)\setminus \cup_j D_{j;\lambda}.
\end{equation}
(Here $A=e^{\left(\frac{e+1}{e-1}\right)^2}$.)
\end{Proposition}
\begin{proof}
Recall that
$\lambda=\bigl((c_{ki})_{0\le i\le p}, (d_{kj})_{1\le j\le q}\bigr)_{1\le k\le m}$, where coordinates $c_{ki}$ and $d_{kj}$ are coefficients of polynomials $P_k$ and $Q_k$, respectively. Consider the compact set
\[
K_1:=\{\lambda\in\Co^{N_{p,q,m}}\, :\, \sum_{k,i}|c_{ki}|=1,\ \max_{1\le k\le m}\sup_{\Di}|Q_k|=1\}\Subset\Co^{N_{p,q,m}}.
\]
We apply property (h$'$) of cyclicity (see \eqref{cartan}) to $K_1$ and $\varepsilon=1$, and set $R_{\mathcal F}:=R_1$ (here $R_\varepsilon$ is as in property (h)). Then from the corresponding inequality \eqref{cartan} using that $A^{-1}\left(\frac{H}{A}\right)^{c(\mathcal F;K)}\ge \left(\frac{H}{A}\right)^{c_{p,q,m}+1}$ we get inequality \eqref{cartan1} for $w=0$, $R_{\lambda;w}=1$.
The general case is easily reduced to the previous one by the substitution $z\mapsto R_{\lambda;w}(z+w)$, $z\in\Co$. We leave the details to the readers.
\end{proof}
If $f_\lambda$ is a polynomial, then its degree is at most $p$ and it satisfies the analog of inequality \eqref{cartan1} with $c_{p,q,m}+1$  replaced by $p$ for all $R>0$.

Similarly, property (h$''$) of cyclicity (see \eqref{remez}) and the arguments of the proof of the previous proposition yield Remez-type inequalities for the family of generalized exponential polynomials.
\begin{Proposition}\label{prop3.5}
There exists an absolute constant $\hat{c}\ge 1$ such that for each nonpolynomial $f_\lambda\in\mathcal F$, an interval $I\subset\Di_{(R_{\lambda;w}R_{\mathcal F})/e}(w)$ and a measurable subset $\omega\subset I$ the following inequality holds:
\begin{equation}\label{remez1}
\sup_I |f_\lambda|\le \Phi\left(\frac{2m_1(I)}{m_1(\omega)}-1\right)^{\hat{c}\cdot\max\{1,c_{p,q,m}\}}\cdot\sup_\omega |f_\lambda|.
\end{equation}
\end{Proposition}
If $f_\lambda$ is a polynomial, then instead of \eqref{remez1} we have the classical Remez inequality with the factor on the right-hand side replaced by $T_p\left(\frac{2m_1(I)}{m_1(\omega)}-1\right)$, where $T_p$ is the Chebyshev polynomial of degree $p$, valid for all measurable $\omega\subset I\Subset\Co$.
\begin{R}
{\rm
Using inequality \eqref{remez1} one obtains various local distributional inequalities for the family of multidimensional generalized exponential polynomials of 
the form $\sum_{k=1}^m P_k e^{Q_k}$, where $P_k$ and $Q_k$ are holomorphic polynomials on $\Co^N$ of degrees $p$ and $q$, respectively (for the corresponding references and results see the Introduction and Section 2 in \cite{Br2}).}
\end{R}

Theorem \ref{main} allows to obtain effective estimates for the cyclicity of more complicated families of functions (for instance, the family $\sum_{k=1}^m\mathcal F_{k1}e^{\mathcal F_{k2}}$, where all $\mathcal F_{ki}$ are families of generalized exponential polynomials). Using then the properties of cyclicity one obtains for such families results similar to those described in the present section.

\section{Proof of Theorem \ref{main}}
\subsection{Equivalent Definition of Cyclicity}
For a domain $O\Subset V\subset\Co^n$ by $\mathcal O_c(\Di,O)$ we denote the set of holomorphic maps $\varphi:\bar{\Di}\rightarrow O$ (i.e., each $\varphi$ is holomorphic in a suitable open neighbourhood of the closure $\bar{\Di}$ of $\Di$.) Let $\mathcal F=\{f_\lambda\}_{\lambda\in V}$ be the family of holomorphic functions in the disk $\Di_\rho$ depending holomorphically on $\lambda\in V$ and $\varphi\in \mathcal O_c(\Di,O)$ be such that $\varphi(\Di)\not\subset\mathcal C(\mathcal F;O)$. We will assume that $\mathcal F\ne 0$, i.e., it contains nonidentically zero functions. Then the family $\mathcal F_\varphi=\{f_{\varphi(w)}\}_{w\in\Di\setminus\varphi^{-1}(\mathcal C(\mathcal F;O))}$ consists of nonidentically zero holomorphic functions on $\Di_\rho$ and its center set $\mathcal C(\mathcal F_\varphi;\Di)=\varphi^{-1}(\mathcal C(\mathcal F;O))$ consists of finitely many points.
Let us consider the Maclaurin series expansion of $\mathcal F_{\varphi}$,
\[
f_{\varphi(w)}(z)=\sum_{k=0}^\infty c_{k}(w)z^k,\quad z\in\Di_\rho,\quad c_{k}\in\mathcal O(\Di).
\]
Let $b_{\varphi}(\mathcal F)\in\Z_+$ be the minimal number such that the ideal $\mathcal I(\mathcal F_\varphi;\Di)\subset\mathcal O(\Di)$ generated by all the $c_{k}$ coincides with the ideal $\mathcal I_{b_{\varphi}(\mathcal F)}(\mathcal F_\varphi;\Di)\subset\mathcal O(\Di)$ generated by $c_{0},\dots, c_{b_{\varphi}(\mathcal F)}$. Then Theorem 1.3 of \cite{Br1} states that for a compact subset $K\subset V$
\begin{equation}\label{eq2.3}
c(\mathcal F;K)=\lim_{O\rightarrow K}\sup_{\varphi\in \mathcal O_c(\Di,O)}b_{\varphi}(\mathcal F),
\end{equation}
where the limit is taken over the filter of open neighbourhoods of $K$.
\subsection{Proof of Theorem \ref{main}}
According to property (a) of cyclicity, it suffices to prove the result for $K=\{\mu\}$, a point in $V$. Without loss of generality we assume that $\mathcal F\ne 0$. Thus, due to \eqref{eq2.3},
there exists $\varphi\in \mathcal O_c(\Di,O)$, $\varphi(\Di)\not\subset\mathcal C(\mathcal F;O)$, where $O$ is an open neighbourhood of $\mu$ such that
\[
c(\mathcal F;\mu)=b_{\varphi}(\mathcal F).
\]
Recall that $\mathcal F:=\sum_{j=1}^k \mathcal F_j$, where $\mathcal F_j:=\{f_{\lambda;j}\}_{\lambda\in V}$, $1\le j\le k$.
By $S_\lambda\subset\mathcal O(\Di_\rho)$ we denote the vector space generated by $f_{\lambda;1},\dots, f_{\lambda;k}$, $\lambda\in V$, and by
$d_\lambda$ its complex dimension, so that $0\le d_\lambda\le k$. 
We set
\[
d_\varphi:=\sup_{w\in\Di}d_{\varphi(w)}.
\]
By the definition, $d_\varphi\ge 1$ and there exist
$w_0\in\Di$ and $1\le j_1<\cdots< j_{d_\varphi}\le k$ such that $S_{\varphi(w_0)}\subset\mathcal O(\Di_\rho)$ coincides with the space generated by $f_{\varphi(w_0);j_s}$, $1\le s\le d_{\varphi}$, and $d_{\varphi(w_0)}=d_\varphi$.

We set 
\[
g_{w;s}:=f_{\varphi(w);j_s},\qquad  W_{w;p}:=W(g_{w;1},\dots, g_{w;s}),\qquad 1\le s\le d_\varphi,\quad\text{and}\quad W_{w;0}:=1.
\]
Then the center set of the family $\{W_{w;d_\varphi}\}_{w\in\Di}$ consists of finitely many points and outside these points each $g_w:=\sum_{j=1}^k f_{\varphi(w);j}$ satisfies the ordinary differential equation with meromorphic coefficients (the Frobenius formula, see \cite{F}):
\begin{equation}\label{frob}
\frac{W_{w;d_\varphi}}{W_{w;d_\varphi-1}}\cdot\frac{d}{dz}\cdot \frac{W_{w;d_\varphi-1}^2}{W_{w;d_\varphi}\cdot W_{w;d_\varphi-2}}\cdot\frac{d}{dz}\cdots\frac{d}{dz}\cdot
\frac{W_{w;1}^2}{W_{w;2}\cdot W_{w;0}}\cdot\frac{d}{dz}\cdot \frac{W_{w;0}}{W_{w;1}}\cdot g_w=0.
\end{equation}

Next, consider the Taylor series expansions of families $\{g_w\}$ and $\{W_{w;s}\}$, $1\le s\le d_\varphi$,
\[
\begin{array}{l}
\displaystyle
g_w(z)=\sum_{k=0}^\infty c_{k}(w)z^k,\quad z\in\Di_\rho,\quad c_{k}\in\mathcal O(\Di);\medskip\\
\displaystyle W_{w;s}(z)=\sum_{k=0}^\infty b_{k;s}(w)z^k,\quad z\in\Di_\rho,\quad b_{k;s}\in\mathcal O(\Di),\quad 1\le s\le d_\varphi.
\end{array}
\]
By definition, the center sets of these families are subsets of the center set, say $C$, of the family $\{W_{w;d_\varphi}\}_{w\in\Di}$. In particular, they are finite subsets of $\Di$. Thus, there exist univariate holomorphic polynomials $P$ and $P_s$ with zeros (if any) in $C$, functions $\tilde c_k$ and $\tilde b_{k;s}\in \mathcal O(\Di)$ and numbers $l_s\in\Z_+$, $1\le s\le d_\varphi$, such that 
$c_k=P\cdot\tilde c_k$ and $b_{k;s}=P_s\cdot\tilde b_{k;s}$ for all $k$ and $s$, and each of 
the families $\{\tilde c_0,\dots, \tilde c_{b_\varphi(\mathcal F)}\}$ and
$\{\tilde b_{0;s},\dots, \tilde b_{l_s;s}\}$, $1\le s\le d_\varphi$, has no common zeros on $\Di$. 
(Therefore by the corona theorem, see, e.g., \cite{GR}, the ideals generated by these families in $\mathcal O(\Di)$ coincide with $\mathcal O(\Di)$.) 

We set
\[
\begin{array}{l}
\displaystyle
\tilde g_w(z)=\sum_{k=0}^\infty \tilde c_{k}(w)z^k,\quad z\in\Di_\rho;\medskip\\
\displaystyle \widetilde W_{w;s}(z)=\sum_{k=0}^\infty \tilde b_{k;s}(w)z^k,\quad z\in\Di_\rho,\quad 1\le s\le d_\varphi.
\end{array}
\]
Then \eqref{frob} implies
\begin{equation}\label{frob1}
\frac{\widetilde W_{w;d_\varphi}}{\widetilde W_{w;d_\varphi-1}}\cdot\frac{d}{dz}\cdot \frac{\widetilde W_{w;d_\varphi-1}^2}{\widetilde W_{w;d_\varphi}\cdot \widetilde W_{w;d_\varphi-2}}\cdot\frac{d}{dz}\cdots\frac{d}{dz}\cdot
\frac{\widetilde W_{w;1}^2}{\widetilde W_{w;2}\cdot W_{w;0}}\cdot\frac{d}{dz}\cdot \frac{W_{w;0}}{\widetilde W_{w;1}}\cdot \tilde g_w=0.
\end{equation}

Further, by the definition of $b_\varphi(\mathcal F)$, there exists $w_*\in\Di$ such that $\tilde c_{b_\varphi(\mathcal F)}(w_*)\ne 0$ but
$\tilde c_k(w^*)=0$ for all $k<b_\varphi(\mathcal F)\, (=c(\mathcal F;\mu))$. In what follows, by $m_s$ we denote the multiplicity of zero of $\widetilde W_{w_*;s}$ at $0\in\Di$, $1\le s\le d_\varphi$. By the definition of cyclicity (cf. \cite[Th.~1.3]{Br1}) $m_s\le c\bigl(W(\mathcal F_{j_1},\dots,\mathcal F_{j_s});\mu\bigr)$,  $1\le s\le d_\varphi$.

We are ready to prove the theorem. Assume, on the contrary, that (cf. the formulation of the theorem)
\begin{equation}\label{eq2.6}
c(\mathcal F;\mu)>\max_{I\in\mathcal K}\{c(\mathcal F_I;\mu)+|I|-1\}.
\end{equation}
We apply the operator on the left-hand side \eqref{frob1} with index $w_*$ to $\tilde g_{w_*}$ and compute the resulting multiplicity of zero at $0\in\Di$. By our hypothesis \eqref{eq2.6}, the multiplicity of zero at $0\in\Di$ of 
\[
\frac{W_{w_*;0}}{\widetilde W_{w_*;1}}\cdot \tilde g_{w_*}
\]
is at least $c(\mathcal F;\mu)-m_1\ge c(\mathcal F;\mu)-c(\mathcal F_{j_1};\mu)\ge 1$. Similarly,
the multiplicity of zero at $0\in\Di$ of 
\[
\frac{\widetilde W_{w_*;1}^2}{\widetilde W_{w_*;2}\cdot W_{w_*;0}}\cdot\frac{d}{dz}\cdot \frac{W_{w_*;0}}{\widetilde W_{w_*;1}}\cdot \tilde g_{w_*}
\]
is at least $(c(\mathcal F;\mu)-m_1-1)+(2m_1-m_2)=c(\mathcal F;\mu)+m_1-m_2-1\ge c(\mathcal F;\mu)-c\bigl(W(\mathcal F_{j_1},\mathcal F_{j_2});\mu\bigr)-1\ge 1$, etc. After $d_\varphi$ steps we obtain that the multiplicity of zero at $0\in\Di$ of
\begin{equation}\label{eq2.7}
\frac{\widetilde W_{w_*;d_\varphi-1}^2}{\widetilde W_{w_*;d_\varphi}\cdot \widetilde W_{w_*;d_\varphi-2}}\cdot\frac{d}{dz}\cdots\frac{d}{dz}\cdot
\frac{\widetilde W_{w_*;1}^2}{\widetilde W_{w_*;2}\cdot W_{w_*;0}}\cdot\frac{d}{dz}\cdot \frac{W_{w_*;0}}{\widetilde W_{w_*;1}}\cdot \tilde g_{w_*}
\end{equation}
is at least $c(\mathcal F;\mu)+m_{d_\varphi -1}-m_{d_\varphi}-d_{\varphi}+1\ge c(\mathcal F;\mu)-c\bigl(W(\mathcal F_{j_1},\dots,\mathcal F_{j_{d_\varphi}});\mu\bigr)-d_{\varphi}+1\ge 1$. On the other hand, due to \eqref{frob1} the function in \eqref{eq2.7} is constant,
a contradiction with the previous conclusion.

The proof of the theorem is complete.

\end{document}